\documentclass{amsproc}
\usepackage[utf8]{inputenc}
\usepackage[T1]{fontenc}
\usepackage{amsmath,amssymb,amsthm}
\usepackage{url}
\usepackage[hidelinks]{hyperref}
\hypersetup{
  pdftitle={When is a polynomial in three variables cylindrical?},
  pdfauthor={Xiaojin Lin}
}

\newtheorem{theorem}{Theorem}
\newtheorem{lemma}[theorem]{Lemma}
\newtheorem{corollary}[theorem]{Corollary}
\newtheorem{proposition}[theorem]{Proposition}
\theoremstyle{definition}
\newtheorem{remark}[theorem]{Remark}

\newcommand{\C}{\mathbb{C}}
\newcommand{\R}{\mathbb{R}}
\newcommand{\PP}{\mathbb{P}}
\newcommand{\Hess}{\operatorname{Hess}}
\newcommand{\adj}{\operatorname{adj}}
\newcommand{\rank}{\operatorname{rank}}
\DeclareMathOperator{\Sing}{Sing}

\title[When is a polynomial in three variables cylindrical?]{When is a polynomial in three variables cylindrical?}

\author{Xiaojin Lin}
\address{Yau Mathematical Sciences Center, Tsinghua University, Beijing 100084, People's Republic of China}
\email{xj-lin@mail.tsinghua.edu.cn}
\thanks{ORCID: \url{https://orcid.org/0009-0009-8681-2096}}


\begin{document}

\begin{abstract}
We prove that a nonconstant polynomial $f\in\C[x_1,x_2,x_3]$ is
cylindrical (that is, $f\in\C[x_1,x_2]$ after an invertible linear
change of coordinates) if and only if its bordered Hessian determinant
vanishes identically. The same conclusion holds over $\R$.
We also give explicit counterexamples to this criterion in every
dimension $n\ge4$.
\end{abstract}

\maketitle

\section{Introduction}

Let $f\in\C[x_1,x_2,x_3]$ be a nonconstant polynomial, $\nabla f$ its
gradient and $\Hess f$ its Hessian matrix. We call $f$
\emph{cylindrical} if it is independent of one variable after an
invertible linear change of coordinates. Following Fox \cite[\S2.1]{Fox}, set
\begin{equation}\label{eq:U}
U(f)\;=\;(\nabla f)^{T}\,\adj(\Hess f)\,\nabla f,
\end{equation}
where $\adj$ denotes the adjugate matrix. Equivalently, $U(f)$ is minus
the determinant of the \emph{bordered Hessian} (see Lemma~\ref{lem:bordered})
\begin{equation}\label{eq:bordered}
B(f)\;=\;
\begin{pmatrix}
\Hess f & \nabla f\\
(\nabla f)^{T} & 0
\end{pmatrix},
\qquad U(f)=-\det B(f).
\end{equation}
At a regular point of a level surface $\{f=c\}$, the condition
$U(f)=0$ is equivalent to degeneracy of its second fundamental form
(Lemma~\ref{lem:geom}). Thus $U(f)\equiv0$ makes every smooth level
surface \emph{developable}.
Fox \cite[\S2.1]{Fox} traces this regular-level-set criterion to Reilly.
Local smooth functions with $U=0$ need not be cylindrical
(Remark~\ref{rem:smooth}). For polynomials in three variables, however,
$U=0$ characterizes cylindricity.

\begin{theorem}[Main theorem]\label{thm:main}
Let $f\in\C[x_1,x_2,x_3]$ be a nonconstant polynomial. The following
conditions are equivalent:
\begin{enumerate}
\item[\textup{(i)}] $U(f)\equiv0$;
\item[\textup{(ii)}] there exists $0\neq w\in\C^{3}$ such that
$D_{w}f:=w\cdot\nabla f\equiv0$;
\item[\textup{(iii)}] after a change of coordinates in
$\mathrm{GL}_{3}(\C)$, $f\in\C[x_1,x_2]$.
\end{enumerate}
\end{theorem}

No hypothesis on the critical locus of $f$ is imposed; in particular $f$
may have arbitrarily singular fibers.

Three remarks place the statement in context.
First, the classical Hesse condition $\det\Hess f\equiv0$ does
\emph{not} force cylindricity: $f=x_1x_2+x_3$ has singular Hessian but
no nonzero constant translation direction, and $U(f)=-1$.
Since cylinders have singular Hessian, Theorem~\ref{thm:main} shows
that $U\equiv0$ is strictly stronger than $\det\Hess\equiv0$ in three
variables.

Second, assume $U(f)=0$. The proof uses Bertini--Krull finiteness
(Theorem~\ref{thm:BK}) and Piontkowski's rank-one cylinder theorem
(Theorem~\ref{thm:piontkowski}). The algebraic point is the passage
from one fiber to the polynomial. Write $f=g\circ h$ with $h$
indecomposable. The composition identity (Lemma~\ref{lem:comp})
gives $U(h)=0$, so a smooth irreducible fiber $h=c$ is cylindrical.
Its direction $w$ gives $h-c\mid D_wh$, and degree forces $D_wh=0$.
The chain rule then gives $D_wf=0$.
Thus one general fiber suffices, with no restrictions on the
critical locus or on special fibers (Section~\ref{sec:proof}).

Third, the dimension bound is sharp. For every $n\ge4$ we
exhibit a polynomial $G_n\in\C[x_1,\dots,x_n]$ with
$\nabla G_n\neq0$ everywhere, $U(G_n)\equiv0$, and no constant
translation direction (Proposition~\ref{prop:sharp}). Affinely smooth
non-conical developable hypersurfaces of Gauss rank two already occur
in $\C^4$; see \cite[Introduction]{P} and \cite[\S2]{AG}.

A real form can be derived from the main theorem~\ref{thm:main}.

\begin{corollary}[Real form]\label{cor:real}
Let $f\in\R[x_1,x_2,x_3]$ be nonconstant with $U(f)\equiv0$. Then there
exists $0\neq w\in\R^{3}$ with $D_w f\equiv0$; equivalently, after a
$\mathrm{GL}_{3}(\R)$ change of coordinates, $f\in\R[x_1,x_2]$.
\end{corollary}

\section{Algebraic preliminaries}\label{sec:prelim}

We record the algebraic identities, the geometric interpretation of $U$,
and the two results used in the proof.

\subsection{The bordered Hessian}

For $f\in\C[x_1,\dots,x_n]$ define $B(f)$ as in \eqref{eq:bordered} and
$U(f)$ as in \eqref{eq:U}.

\begin{lemma}\label{lem:bordered}
$U(f)=-\det B(f)$.
\end{lemma}

\begin{proof}
For a parameter $s$, the determinant of $H+sI$ (where $H:=\Hess f$) is a
nonzero polynomial in $s$, so $H+sI$ is invertible over the function
field. The Schur complement formula then gives
\[
\det\begin{pmatrix}H+sI & \nabla f\\ (\nabla f)^{T}&0\end{pmatrix}
=-(\nabla f)^{T}\adj(H+sI)\,\nabla f .
\]
Both sides are polynomial in $s$ and in the entries of $H$; setting
$s=0$ yields $\det B(f)=-U(f)$ for every Hessian, including the case
$\det\Hess f\equiv0$.
\end{proof}

For $A\in\mathrm{GL}_n(\C)$,
\[
B(f\circ A)(x)=
\begin{pmatrix}A^T&0\\0&1\end{pmatrix}
B(f)(Ax)
\begin{pmatrix}A&0\\0&1\end{pmatrix},
\]
and hence
\[
U(f\circ A)(x)=(\det A)^2U(f)(Ax).
\]
Thus the vanishing of $U$ is invariant under linear changes of coordinates.

\subsection{The geometric meaning of $U$}

\begin{lemma}\label{lem:geom}
Let $f\in\C[x_1,x_2,x_3]$ and let $p$ be a regular point of the level
surface $X=\{f=c\}$. Then $U(f)(p)=0$ if and only if the second
fundamental form of $X$ at $p$ is degenerate; equivalently, the Gauss
map of $X$ has rank $\le1$ at $p$.
\end{lemma}

\begin{proof}
After a linear change of coordinates in $\mathrm{GL}_3(\C)$ chosen so
that the $(x_1,x_2)$-plane spans $\ker df(p)$, we may assume
$\nabla f(p)=(0,0,f_3)$ with $f_3\neq0$. Writing $H=(f_{ij})$, only the
last row of the adjugate contributes:
\[
U(f)(p)=\big(0,0,f_3\big)\,\adj(H)\,
\big(0,0,f_3\big)^{T}
=f_3^{2}\,\det\begin{pmatrix}f_{11}&f_{12}\\ f_{12}&f_{22}\end{pmatrix}\!(p).
\]
The tangent space $T_pX=\ker df(p)$ is the $(x_1,x_2)$-plane, and the
(affine or projective) second fundamental form of $X$ at $p$ is
represented by the restriction of $\Hess f(p)$ to $T_pX$, divided by
the nonzero scalar $f_3(p)$. The differential of the projective Gauss
map of the projective closure of $X$ has the same rank as this second
fundamental form. No Euclidean metric is involved.
Hence $U(f)(p)=0$ if and only if
$\det II_p=0$, which for a surface means $\rank II_p\le1$, i.e.\ the
Gauss map drops rank at $p$.
\end{proof}

Consequently, if $U(f)\equiv0$ then every smooth level surface of $f$
is \emph{developable}: its Gauss rank is at most one at every point.

\subsection{A composition identity}

We use the following polynomial form of Fox's reparametrization identity
\cite[\S2.2]{Fox}.

\begin{lemma}\label{lem:comp}
Let $h\in\C[x_1,\dots,x_n]$ and let $g\in\C[t]$ be nonconstant. Then
\[
U(g\circ h)=\big(g'(h)\big)^{n+1}\,U(h).
\]
\end{lemma}

\begin{proof}
If $h$ is constant, both sides vanish. Assume henceforth that $h$ is
nonconstant, so $g'(h)\not\equiv0$.
Write $a=g'(h)$, $b=g''(h)$, $H=\Hess h$, $v=\nabla h$. Then
$\nabla(g\circ h)=a\,v$ and
$\Hess(g\circ h)=aH+b\,vv^{T}$, so
\[
B(g\circ h)=
\begin{pmatrix}
aH+b\,vv^{T} & av\\
av^{T} & 0
\end{pmatrix}.
\]
On the Zariski-open set $\{a\neq0\}$, subtracting $(bv_i/a)$ times the
last row from the $i$-th row ($i=1,\dots,n$) eliminates the $b\,vv^{T}$
block and leaves the matrix
$\begin{pmatrix}aH&av\\ av^{T}&0\end{pmatrix}=a\,B(h)$, whose
determinant is $a^{n+1}\det B(h)$. By Lemma~\ref{lem:bordered} this
gives $U(g\circ h)=a^{n+1}U(h)$ on $\{a\neq0\}$; both sides are
polynomials, so the identity holds everywhere.
\end{proof}

\subsection{Two quoted theorems}

Recall that $h\in\C[x_1,\dots,x_n]$ is \emph{indecomposable} if it
cannot be written as $g\circ h_0$ with $\deg g\ge2$.

\begin{theorem}[Bertini--Krull]\label{thm:BK}
Let $h\in\C[x_1,\dots,x_n]$, $n\ge2$, be indecomposable. Then $h-c$ is
irreducible for all but finitely many $c\in\C$; in fact,
the number of reducible fibers is at most $\deg h-1$.
\end{theorem}

For the finiteness statement, see \cite[Theorem~37]{Sch} and
\cite[\S1]{BDN}. The bound is the Stein--Lorenzini--Najib inequality
\cite{St,Naj}.

\begin{theorem}[Piontkowski {\cite[Corollary~5]{P}}]\label{thm:piontkowski}
Let $X\subset\PP^{N}$ be an irreducible developable variety of Gauss
rank one, and let $H_\infty\subset\PP^{N}$ be a hyperplane with
$X\not\subset H_\infty$ and $\Sing(X)\subset H_\infty$.
Then $X$ is a cone whose vertex lies in $H_\infty$.
In particular, if $N=3$, the affine part
$X\cap(\PP^{3}\setminus H_\infty)$ is invariant under translations in a
fixed nonzero direction $w\in\C^{3}$.
\end{theorem}

The last statement is the standard projective-to-affine correspondence:
a cone whose vertex lies in the hyperplane at infinity restricts to an
affine cylinder (cf.\ \cite[\S6]{AG}). A developable surface of Gauss
rank zero is a plane, which is also a cylinder.
Related differential-geometric cylinder theorems are due to
Hartman--Nirenberg \cite[Theorem~III]{HN} and Abe \cite{Abe}.

\section{Proof of the main theorem}\label{sec:proof}

The implications (iii)$\Rightarrow$(ii) and (ii)$\Rightarrow$(i) are
elementary. Indeed, (iii) means that $\partial f/\partial x_3=0$ in
suitable coordinates, which is (ii) with $w=e_3$; conversely, sending
$w$ to $e_3$ by an element of $\mathrm{GL}_3(\C)$ turns $D_w f=0$ into
$\partial f/\partial x_3=0$, which integrates to $f\in\C[x_1,x_2]$. If
$D_w f=0$, then differentiating gives $\Hess f\cdot w=\nabla(D_wf)=0$,
so $\rank\Hess f\le2$ everywhere; if the rank is $\le1$ then
$\adj\Hess f=0$, and if it is $2$ then
$\adj\Hess f=\lambda\, ww^{T}$ up to sign, whence
$U(f)=\lambda\,(w\cdot\nabla f)^{2}=0$. This proves (i).

The content is (i)$\Rightarrow$(ii). Assume $U(f)\equiv0$.

\smallskip
\noindent\emph{Step 1: reduction to an indecomposable generator.}
Write $f=g\circ h$ with $h\in\C[x_1,x_2,x_3]$ indecomposable and
$g\in\C[t]$ (iterate any nontrivial decomposition; degrees decrease, so
the process terminates). If $\deg h=1$, then $h$ is affine-linear and
$f$ already has the form (iii), so we are done. Assume
$\deg h\ge2$. By Lemma~\ref{lem:comp},
$0=U(f)=(g'(h))^{4}\,U(h)$; since $g'\not\equiv0$, we conclude
$U(h)\equiv0$.

\smallskip
\noindent\emph{Step 2: a good fiber.}
By Theorem~\ref{thm:BK} the fiber $X_c=\{h=c\}$ is irreducible for all
but finitely many $c$, and by generic smoothness over $\C$ the fiber is
smooth for all but finitely many $c$. Fix $c$ outside both finite sets.
Then $X_c$ is a smooth irreducible affine surface, and $U(h)\equiv0$
implies, by Lemma~\ref{lem:geom}, that the second fundamental form of
$X_c$ is degenerate at every point: $X_c$ is developable, of Gauss rank
$\le1$.

\smallskip
\noindent\emph{Step 3: the fiber is a cylinder.}
Let $\overline{X_c}\subset\PP^{3}$ be the projective closure and
$H_\infty=\PP^{3}\setminus\C^{3}$. Since $X_c$ is smooth,
$\Sing(\overline{X_c})\subset H_\infty$. If the Gauss rank of
$\overline{X_c}$ is zero, $\overline{X_c}$ is a plane. Otherwise the
Gauss rank is one on a dense open set, and Theorem~\ref{thm:piontkowski}
makes $\overline{X_c}$ a cone whose vertex is a point of $H_\infty$.
In both cases there is a nonzero vector $w\in\C^{3}$ such that $X_c$ is
invariant under the translations $p\mapsto p+tw$, $t\in\C$.

\smallskip
\noindent\emph{Step 4: the axis annihilates $h$ globally.}
Since $h$ is constant on $X_c$ and $X_c$ is $w$-invariant, the
polynomial $D_w h$ vanishes on $X_c$. Because $h-c$ is irreducible and
$X_c$ is reduced, this gives $(h-c)\mid D_w h$. But
\[
\deg D_w h\;\le\;\deg h-1\;<\;\deg(h-c),
\]
so divisibility forces $D_w h\equiv0$.

\smallskip
\noindent\emph{Step 5: descent to $f$.}
$D_w f=g'(h)\cdot D_w h\equiv0$, which is (ii). \qed

\begin{remark}\label{rem:strategy}
The degree argument in Step~4 works in every dimension. The obstruction
is Step~3: affinely smooth developable hypersurfaces need not be cones
in higher dimensions \cite[Introduction]{P}. At regular points, $U=0$ implies
Gauss rank at most $n-2$, which for $n\ge4$ need not be at most one.
The rank-one cylinder theorem therefore does not apply in general.
\end{remark}

\pagebreak
\section{Consequences and sharpness}

\subsection{The real form}

\begin{proof}[Proof of Corollary~\ref{cor:real}]
Apply Theorem~\ref{thm:main} to $f$ viewed as a complex polynomial:
there is $0\neq w\in\C^{3}$ with $D_w f=0$. Write $w=a+ib$ with
$a,b\in\R^{3}$. Then $D_wf=D_af+i\,D_bf$, and since $D_af$ and $D_bf$
have real coefficients, $D_wf\equiv0$ implies $D_af\equiv0$ and
$D_bf\equiv0$. At least one of $a,b$ is nonzero.
\end{proof}

\subsection{Sharpness in every dimension $n\ge4$}

For $n\ge4$ define
\[
G_n\;=\;x_1^{2}x_2+x_1x_3+x_4+\sum_{j=5}^{n}x_j^{2}
\;\in\;\C[x_1,\dots,x_n].
\]

\begin{proposition}\label{prop:sharp}
For every $n\ge4$: \textup{(a)} $\nabla G_n$ has no zeros on $\C^{n}$;
\textup{(b)} $U(G_n)\equiv0$; \textup{(c)} $G_n$ admits no nonzero
constant translation direction, and in particular is not cylindrical.
\end{proposition}

\begin{proof}
(a) $\partial G_n/\partial x_4=1$. (b) The Hessian of $G_n$ is block
diagonal with the quadratic part contributing $2I_{n-4}$ and
the remaining block supported on the coordinates $x_1,x_2,x_3$, where
\[
\Hess_{x_1,x_2,x_3}(x_1^{2}x_2+x_1x_3)=
\begin{pmatrix}
2x_2 & 2x_1 & 1\\
2x_1 & 0 & 0\\
1 & 0 & 0
\end{pmatrix}
\]
has rank $\le2$ (its third row and column force the determinant to
vanish). Hence $\rank\Hess G_n\le 2+(n-4)=n-2$ everywhere, so
$\adj\Hess G_n=0$ and $U(G_n)\equiv0$. (c) For
$w=(a_1,\dots,a_n)$,
\[
D_wG_n=a_1(2x_1x_2+x_3)+a_2x_1^{2}+a_3x_1+a_4
+2\sum_{j\ge5}a_jx_j.
\]
Reading off the coefficients of $x_1^{2}$, $x_1x_2$, $x_3$, $x_1$, the
constant term and the $x_j$ ($j\ge5$) in this order forces $w=0$.
\end{proof}

Thus the implication (i)$\Rightarrow$(ii) of Theorem~\ref{thm:main}
fails in every dimension $n\ge4$, even for polynomials without critical
points.

\subsection{Comparisons}

\begin{remark}[Hesse and Gordan--N\"other]
For \emph{homogeneous} forms
in three variables the Gordan--N\"other theorem \cite{GN} does force a
cone from $\det\Hess\equiv0$. De~Bondt--van~den~Essen \cite{BE} give
normal forms for three-variable polynomials with singular Hessian.
Affine-linear terms do not change the Hessian but may destroy every
constant translation direction, as $x_1x_2+x_3$ shows. Their
classification alone therefore does not imply Theorem~\ref{thm:main};
the bordered invariant also records the gradient.
\end{remark}

\begin{remark}[The smooth category]\label{rem:smooth}
Local $C^\infty$ rigidity fails. On the open set $\{y\neq0\}$, the
function
\[
F(x,y,z)=z-\frac{x^{2}}{2y}
\]
satisfies $U(F)=0$: writing $F=z-g(x,y)$ with $g=x^{2}/(2y)$, one has
$\det\Hess g=0$, and for a graph one computes
$U(z-g)=\det\Hess g$. On the other hand,
$aF_x+bF_y+cF_z\equiv0$ forces $a=b=c=0$, so $F$ has no nonzero
constant translation direction; its level surfaces are open subsets of
quadratic cones. The polynomial rigidity in Theorem~\ref{thm:main}
relies on Bertini--Krull finiteness and the degree comparison of Step~4,
which have no smooth analogues. For the equiaffine geometry of smooth
level sets, see Fox \cite[\S2]{Fox}. It would be interesting to know
whether the conclusion persists for entire functions of finite order.
\end{remark}

\begin{remark}[Positive characteristic]
The proof uses characteristic zero, in particular in the composition
reduction and the Gauss-map argument. In positive characteristic, a
degenerate Gauss map need not yield the ruled developability used here;
see \cite[Introduction and Example~7.2]{Fuk}.
\end{remark}

\medskip
\noindent\textbf{Use of AI tools.}
Kimi helped locate Piontkowski's paper \cite{P}. ChatGPT (OpenAI) helped
identify errors in an early draft and simplify Step~4 of the proof of
Theorem~\ref{thm:main}. Codex (OpenAI) assisted with exposition and
\LaTeX{} editing, checks of mathematical arguments and references, and
compilation and PDF verification. The author takes responsibility for
the mathematical content and the accuracy of the references.

\end{document}